\documentclass[12pt]{article}
\usepackage{amsthm, amsfonts}
\usepackage{url}
\usepackage{amscd,amsmath,amssymb}

\newtheorem{definition}{Definition}
\newtheorem{theorem}{Theorem}

\newtheorem{proposition}{Proposition}

\def\de{\delta}

\def\eps{\varepsilon}

\def\al{\alpha}

\def\la{\lambda}

\def\kappa{\varkappa}
\def\si{\sigma}

\def\C{{\mathbb C}}
\def\Y{{\mathbb Y}}

\def\sN{{\mathfrak S}_N}
\def\sinf{{\mathfrak S}_{\mathbb N}}

\def\T{{\cal T}}

\def\Todd{{\cal T}^{\rm odd}}
\def\Teven{{\cal T}^{\rm even}}
\def\Tproper{{\cal T}^{\rm proper}}
\def\H{{\cal H}}

\def\End{\operatorname{End}}

\def\beq{\begin{equation}}
\def\eeq{\end{equation}}

\urldef\vershik\url{vershik@pdmi.ras.ru}
\urldef\natalia\url{natalia@pdmi.ras.ru}

\author{N.~V.~Tsilevich\thanks{%
St.~Petersburg Department of Steklov Institute of Mathematics.
E-mail: \natalia, \vershik. Supported by the grants
RFBR 11-01-12092-ofi-m and RFBR 11-01-00677-a.}
\and A.~M.~Vershik\footnotemark[1]}
\title{Infinite-dimensional Schur--Weyl duality and the
Coxeter--Laplace operator}

\begin{document}
\maketitle

\begin{abstract}
We extend the classical Schur--Weyl duality between representations of the groups
$SL(n,\C)$ and $\sN$ to the case of $SL(n,\C)$ and the infinite
symmetric group $\sinf$. Our construction is based on a ``dynamic,''
or inductive, scheme of Schur--Weyl  dualities.
It leads to a new class of representations of the
infinite symmetric group, which have not appeared earlier.
We describe these representations and, in particular, find their
spectral types with respect to the Gelfand--Tsetlin algebra.
The main example of such a representation acts
in an incomplete infinite tensor product. As an important
application, we consider the weak limit of the so-called
Coxeter--Laplace operator, which is essentially the Hamiltonian of the
XXX Heisenberg model, in these representations.
\end{abstract}

\section{Introduction}

\subsection{General setting}
We extend the classical Schur--Weyl duality \cite[Chap.~4, Sec.~4]{Weyl}
between irreducible representations of
the general linear group
$GL(n,\C)$ (or the special linear group $SL(n,\C)$)
and the (finite) symmetric group
$\sN$ to the case of $SL(n,\C)$ and the infinite
symmetric group $\sinf$. Usually, one considers only the ``static''
Schur--Weyl duality, when the parameter $N$ is fixed.
Our construction is based on a ``dynamic'' view of this duality, which
allows us to consider an
inductive scheme and pass to
the limit, obtaining an infinite-dimensional version of the
Schur--Weyl duality.
This construction leads to a new class of
representations of the
infinite symmetric group, which have not appeared earlier.
The main example is the so-called
tensor representation, which is realized in an incomplete infinite tensor
product.

In this paper, we consider only the simplest case $N=2$, since the
case of a general
$N$ can be handled in exactly the same way.

One of our motivations for  considering Schur--Weyl representations was to
study the behavior of the so-called
Coxeter--Laplace operator $L_N$, or the Hamiltonian of the XXX Heisenberg
model, in these representations.
In particular, we show that a generalized Schur--Weyl scheme
allows one to construct a representation in which the weak limit of
$\frac1NL_N$ is a scalar operator with the scalar
arbitrarily close to the maximum possible value $c_{\rm max}$.

Our representations have natural links to representations of the
Virasoro algebra, Glimm algebra, and other important
representation-theoretic objects. Further analysis should clarify
these relations. We would also like to mention the paper
\cite{Penkov}, where another infinite-dimensional generalization of the
Schur--Weyl scheme is developed. The difference is as follows:
starting from the classical
Schur--Weyl duality between $GL(n,\C)$ and $\sN$, we keep $n$ fixed
and send $N$ to infinity, obtaining a duality between $GL(n,\C)$
and $\sinf$; in \cite{Penkov}, on the contrary, $N$ is kept fixed and
$n$ goes to infinity, resulting in a duality between
$\mathfrak{gl}_\infty$ and $\sN$. Another related paper is \cite{FF},
where an inductive construction of representations of the
affine Lie algebra $\widehat{\mathfrak{sl}_2}$ is suggested, which
starts from the tensor representations of ${\mathfrak{sl}_2}$ and uses
the notion of fusion product of representations.

\subsection{Main results}
Now let us describe our results in more detail.

We consider the representations of $\sinf$ that are the
inductive limits of two-row representations of the finite symmetric
groups under so-called {\it Schur--Weyl embeddings}, which
send a representation of $\sN$ to a representation of ${\mathfrak
S}_{N+2}$ and
respect both the actions of $SL(2,\C)$ and $\sN$. The structure of a
general representation of this kind (which we also call {\it
Schur--Weyl representations} of $\sinf$) is described in
Theorem~\ref{th:reprstr}. Namely,
\beq\label{decomp}
\H=\sum_{k}\Pi_{k}\otimes M_{k+1},
\eeq
where $\Pi_{k}$ is an irreducible representation of $\sinf$
(the inductive limit of a sequence of irreducible
representations of the symmetric groups), $M_k$ is the irreducible
representation of $SL(2,\C)$ of dimension $k$, and the sum is taken over
either odd or even $k$; moreover,
$\Pi_{k}\otimes M_{k+1}$ is an irreducible representation of
$\sinf\times SL(2,\C)$, and the operator algebras generated by the actions
of $\sinf$ and $SL(2,\C)$ are mutual commutants.

Thus it suffices to study the
irreducible representations $\Pi_k$ of $\sinf$
obtained in this way. In particular, we show that the spectral type
of such a representation with respect to the Gelfand--Tsetlin algebra
is determined by a $\sigma$-finite, Bernoulli-like, noncentral measure
on the space of infinite Young tableaux
(Theorem~\ref{th:Bernoulli}).

The most interesting example of a Schur--Weyl representation is the
so-called {\it tensor} representation, obtained from the unique
Schur--Weyl embeddings that preserve the tensor product structure of
the space $(\C)^{\otimes n}$ to which the Schur--Weyl duality applies.
This representation is studied in Section~\ref{product}. It can be
realized in the incomplete tensor product of the spaces $\C^4$, the
distinguished vector being the unique $SL(2,\C)$-invariant vector in $\C^4$.

Observe an analogy between the decomposition~\eqref{decomp} of a
Schur--Weyl representation of the infinite symmetric group and the
decomposition (a limiting case of the Goddard--Kent--Olive construction)
$$
{\cal M}_j=\sum_{k}L(1,k^2)\otimes M_{k+1},\qquad j=0,1/2,
$$
where ${\cal M}_j$ is the level~$1$ spin~$j$ fundamental representation of the affine Lie
algebra $\widehat{\mathfrak{sl}_2}$, $L(1,k^2)$ is the irreducible
representation of the Virasoro algebra $\mathbf{Vir}$ with central charge~$1$ and
conformal dimension~$k^2$, the summation is over all positive integers
$k$ of the same parity as $j$, and the algebras $\mathbf{Vir}$ and
${\mathfrak{sl}_2}\subset\widehat{\mathfrak{sl}_2}$ are mutual
commutants (see, e.g., \cite{Wasserman}). This analogy suggests that one may introduce a natural
action of the Virasoro algebra in a Schur--Weyl module, or,
equivalently, a natural action of the infinite symmetric group in the
fundamental module ${\cal M}_j$.

The last section of the paper concerns the so-called
{\it periodic Coxeter Laplacian}, or the {\it Coxeter--Laplace operator}.
This is the operator
$L_N=Ne-(s_1+{\ldots} +s_N)$
in the group algebra of the symmetric group $\sN$, where $s_k$ is the
Coxeter transposition $(k,k+1)$ (with
$N+1\equiv 1$). If $\pi_N$ is the standard representation of
$\sN$ in
$(\C^2)^{\otimes N}$,
then the operator $\pi_N(L_N)$ is related to the Hamiltonian of the XXX
Heisenberg model on the periodic one-dimensional lattice with $N$
sites (see, e.g., \cite{Iz,FT}) by the formula $H=\frac J4(2L-N)$,
where $J>0$ corresponds to
the ferromagnetic case, and
$J<0$ to the antiferromagnetic one.
In Section~\ref{laplacian}, we find the ``antiferromagnetic'' weak
limit of the Coxeter--Laplace operator in the stationary Schur--Weyl representations of the
infinite symmetric groups (Proposition~\ref{prop:p}). It is a scalar
operator with constant depending on the parameter of Schur--Weyl
embeddings, and the maximum possible value of this constant is greater
than for all other natural representations of $\sinf$ considered so far.
But it is still smaller than the limiting value $c_{\rm max}$ of the
ground energy first computed in \cite{Hul} and rigorously proved
in \cite{yangyang2}. However, in Proposition~\ref{prop:apprmax} we
show that by extending the construction of Schur--Weyl embeddings one
can build a representation of $\sinf$ with the corresponding constant
arbitrarily close to $c_{\rm max}$.

\subsection{Notation}
Here we present necessary notation from the combinatorics of Young
diagrams and tableaux. Let $\Y_N$ be the set
of Young diagrams with $N$ cells and
$\Y_N^{l}\subset\Y_N$ be the set of Young diagrams with $N$ cells and at
most $l$ rows. Let $\T_N$ be the set of Young tableaux with $N$ cells.
Given $\la\in\Y_N$, let $\T_N(\la)\subset\T_N$ be
the set of Young tableaux with diagram $\la$, that is, the set of
paths in the Young graph from the unique vertex $\emptyset$ of zero
level to $\la$. Given Young
diagrams $\la\in\Y_N$, $\mu\in\Y_{N+k}$ such that $\la\subset\mu$, we
also denote by $\T(\la,\mu)$ the set of paths in the Young graph from
$\la$ to $\mu$, and by $H(\la,\mu)$ the Hilbert space in which the
elements of
$\T(\la,\mu)$ are an orthonormal basis.
By $[t]_N$ we denote the initial segment of length $N$ of a Young tableau
$t$.
Finally, given $\la\in\Y_N$ and $\mu\in\Y_{N+2}$ with at most two rows each, we say that the pair
$(\la,\mu)$ is {\it nice} if $\la\subset\mu$ and $\mu$ is obtained
from $\la$ by adding one cell to each row. Obviously, if a pair
$(\la,\mu)$ is nice, then $\T(\la,\mu)$ contains exactly two elements.

\section{An inductive construction of Schur--Weyl embeddings}
\label{swembeddings}

The classical Schur--Weyl duality (see, e.g., \cite{FultonHarris})
is a fundamental theorem that
relates irreducible representations of the general linear group
$GL(l,\C)$ and the
symmetric group ${\mathfrak S}_N$ in the tensor power $(\C^l)^{\otimes
N}$, where
${\mathfrak S}_N$ permutes the factors and $GL(l,\C)$
acts by the simultaneous matrix multiplication:
$$
(\C^l)^{\otimes N}=\sum_{\la\in\Y_N^{l}}\pi_\la\otimes\rho_\la,
$$
where $\pi_\la$ is the irreducible representation of $\sN$
corresponding to $\la$ and $\rho_\la$ is the irreducible
representation of $GL(l)$ with signature $\la$. The operator algebras
generated by the actions of $\sN$ and $GL(l)$, respectively, are
mutual commutants in the whole operator algebra $\End((\C^l)^{\otimes N})$.

Consider the particular case $l=2$. For definiteness, let $N=2n+1$ be odd. Then
$$
(\C^2)^{\otimes N}=\sum_{\la\in\Y_N^{2}}\pi_\la\otimes\rho_\la
=\sum_{k=0}^{n}\pi_{\la^{(k)}}\otimes\rho_{\la^{(k)}},
$$
where $\la^{(k)}=(\la^{(k)}_1,\la^{(k)}_2)=(n+k+1,n-k)$, so that $\la^{(k)}_1-\la^{(k)}_2=2k+1$.

Consider the restriction of $\rho_{\la^{(k)}}$ to the subgroup
$SL(2,\C)\subset GL(2,\C)$. This is an irreducible representation of
$SL(2,\C)$ that depends
only on the difference $\la^{(k)}_1-\la^{(k)}_2=2k+1$, i.e., on $k$.
Let $M_{m}$ be the irreducible representation of $SL(2,\C)$
of dimension $m$. Then
\beq\label{finodd}
(\C^2)^{\otimes (2n+1)}
=\sum_{j=0}^{n}\pi_{(n+j+1,n-j)}\otimes M_{2j+2}.
\eeq
In a similar way, for even $N=2n$ we have
\beq\label{fineven}
(\C^2)^{\otimes (2n)}
=\sum_{j=0}^{n}\pi_{(n+j,n-j)}\otimes M_{2j+1}.
\eeq

Now consider embeddings $(\C^2)^{\otimes
N}\hookrightarrow(\C^2)^{\otimes (N+2)}$ that preserve this
Schur--Weyl structure. We endow the tensor spaces under consideration
with the standard inner product and regard all representations as
unitary representations. Observe that both in $(\C^2)^{\otimes N}$ and
$(\C^2)^{\otimes (N+2)}$ we have actions of $SL(2,\C)$
and ${\mathfrak S}_{N}$ (with the standard embedding
${\mathfrak S}_{N}\subset{\mathfrak S}_{N+2}$).

\begin{definition}
Isometric embeddings $(\C^2)^{\otimes
N}\hookrightarrow(\C^2)^{\otimes (N+2)}$
that are equivariant with respect to both these actions (in other
words, equivariant  with
respect to the action of $\sN\times SL(2,\C)$)
will be called {\it Schur--Weyl
embeddings}.
\end{definition}

Our first purpose is to describe all
Schur--Weyl embeddings $(\C^2)^{\otimes
N}\hookrightarrow(\C^2)^{\otimes (N+2)}$. Proposition~\ref{embed} below
says that such an embedding
is determined by a sequence of vectors from the one-dimensional
complex circle $\mathbb T^1$.

\begin{proposition}\label{embed}
The Schur--Weyl embeddings $(\C^2)^{\otimes
N}\hookrightarrow(\C^2)^{\otimes (N+2)}$ with $N=2n-1$ or $N=2n$ are
indexed by the elements of
$(\mathbb T^1)^{n}$, where $\mathbb T^1$ is the one-dimensional
complex circle: $\mathbb T^1=\{z\in\C:|z|=1\}$.
\end{proposition}

\begin{proof}
Obviously, under a Schur--Weyl embedding, for each $k$ we have
$$
\pi_{(n+k-1,n-k)}\otimes M_{2k}\hookrightarrow\pi_{(n+k,n-k+1)}\otimes
M_{2k},
$$
where $\pi_{(n+k-1,n-k)}\hookrightarrow\pi_{(n+k,n-k+1)}$ is an
embedding of the irreducible representation $\pi_{(n+k-1,n-k)}$ of
${\mathfrak S}_{2n-1}$ into the irreducible representation
$\pi_{(n+k+1,n-k+1)}$ of
${\mathfrak S}_{2n+1}$. A similar fact holds in the case of an even $N$.

Obviously,
$H(\la,\mu)$ is the multiplicity space of $\pi_\la$ in $\pi_\mu$,
so that the isometric embeddings
$\pi_\la\hookrightarrow\pi_\mu$ commuting with the
action of ${\mathfrak S}_N$ are indexed by the unit vectors $h\in H(\la,\mu)$.
For $\la=(n+k-1,n-k)$ and $\mu=(n+k,n-k+1)$, the pair $(\la,\mu)$ is
nice, so that the set
$\T(\la,\mu)$ consists of two tableaux: the tableau $t_{21}$ is obtained by
putting the element $2n$ into the second row and the element $2n+1$ into
the first row, while the tableau $t_{12}$ is obtained by putting $2n$ into the first
row and $2n+1$ into the second row. Thus we can {\it identify}
the spaces $H((n+k-1,n-k), (n+k,n-k+1))$ for all $k$. Denote the
obtained space by $H^{1,1}$, and fix the standard basis $\{t_{21},t_{12}\}$ of
$H^{1,1}$. Then the
isometric embeddings $\pi_\la\hookrightarrow\pi_\mu$ commuting with the
action of ${\mathfrak S}_N$ are indexed by the elements of the one-dimensional complex circle
$\mathbb T^1\subset H^{1,1}$.

In the case of even $N$, the situation is exactly the same, with the
only exception: for the diagrams $\la=(n,n)$ and $\mu=(n+1,n+1)$, the
multiplicity space $H(\la,\mu)$ is one-dimensional.
\end{proof}

\medskip\noindent{\bf Remark.} Note that a Schur--Weyl embedding
$(\C^2)^{\otimes N}\hookrightarrow(\C^2)^{\otimes (N+2)}$
does not necessarily preserve the tensor product structure. The
important class of Schur--Weyl embeddings that do have this property is
considered in Section~\ref{product}.

\medskip
The scheme described above has a natural
generalization. Namely, we can consider embeddings that preserve the
Schur--Weyl structure but ``jump'' over
an arbitrary (even) number of levels
instead of two ones.

\begin{definition}\label{generalized}
A {\em generalized Schur--Weyl embedding}
is an isometric embedding $(\C^2)^{\otimes
N}\hookrightarrow(\C^2)^{\otimes (N+2k)}$, $k\ge1$, that is
equivariant with respect to
the actions of $SL(2,\C)$ and ${\mathfrak
S}_{N}$  (with the standard embedding
${\mathfrak S}_{N}\subset{\mathfrak S}_{N+2k}$).
\end{definition}

The theory of Schur--Weyl representations
of the infinite symmetric group $\sinf$ developed below can easily be
extended to the case of generalized Schur--Weyl embeddings. In
particular, generalized Schur--Weyl embeddings are used
in Section~\ref{laplacian} for constructing a representation of $\sinf$
in which the weak limit of the so-called Coxeter--Laplace operators has
an eigenvalue that is arbitrarily close to the maximal possible value.

\section{Infinite-dimensional Schur--Weyl duality}
\label{swduality}

Consider an {\it infinite} chain of Schur--Weyl embeddings
\beq\label{infswemb}
(\C^2)^{\otimes0}\hookrightarrow(\C^2)^{\otimes 2}\hookrightarrow(\C^2)^{\otimes
4}\hookrightarrow{\ldots}\quad\mbox{or}\quad(\C^2)^{\otimes1}\hookrightarrow(\C^2)^{\otimes 3}\hookrightarrow(\C^2)^{\otimes
5}\hookrightarrow{\ldots}
\eeq
(in what follows, these two cases will be referred to as ``even'' and
``odd,'' respectively)
and the corresponding inductive limit $\Pi$ of
representations of the symmetric groups.
In the space $\H$ of this
representation we have commuting actions of the infinite symmetric
group $\sinf$ and the special linear group $SL(2,\C)$. By above, the
representation $\Pi$ is determined by a collection of vectors
$h^{(k)}_j\in {\mathbb T}^{1}$, $k=0,1,{\ldots}$,
$j=0,1,{\ldots}$,
where $h^{(k)}_j$ determines the embedding
$\pi_{(k+j,j)}\hookrightarrow\pi_{(k+j+1,j+1)}$.

\begin{theorem}\label{th:reprstr}
Let $\Pi$ be
the representation of the infinite symmetric
group $\sinf$ that is the inductive limit of the standard
representations of $\sN$ in $(\C^2)^{\otimes N}$ with respect to an infinite
chain~\eqref{infswemb} of Schur--Weyl embeddings. Then $\Pi$
decomposes into a countable direct
sum of primary representations
\beq\label{thdecomp}
\Pi=\sum_{k=0}^\infty\Pi_{k}(h^{(k)})\otimes M_{k+1},
\eeq
where $\Pi_{k}(h^{(k)})$ is the inductive limit of the irreducible
representations of the symmetric groups ${\mathfrak S}_{k},
{\mathfrak S}_{k+2},{\ldots} $ corresponding to
the Young diagrams
$$
(k),\quad(k+1,1),\quad(k+2,2),\;{\ldots},\; (k+n,n),\;{\ldots}
$$
determined by the sequence
$h^{(k)}=(h^{(k)}_0,h^{(k)}_1,h^{(k)}_2,{\ldots})\in
(\mathbb T^1)^\infty$,
and the sum is taken over even $k$ in the even case
and over odd $k$ in the odd case.

The representation
$\Pi_{k}(h^{(k)})$ of $\sinf$ is irreducible.
\end{theorem}
\begin{proof}
Follows from the previous considerations.
\end{proof}

\begin{definition}
The representation $\Pi$ of the infinite symmetric group
will be called the Schur--Weyl representation of $\sinf$ determined by
the sequence of Schur--Weyl embeddings with parameters $h^{(k)}_j\in
\mathbb T^1$, $k=0,1,{\ldots}$, $j=0,1,{\ldots}$. Representations
of the form $\Pi_{k}(h^{(k)})$ will be called irreducible Schur--Weyl
representations.
\end{definition}

Thus, if we are interested in the representation theory of the
infinite symmetric group, it suffices to study the irreducible Schur--Weyl
representations $\Pi_k(h)$
for $h=(h_0,h_1,h_2,{\ldots})\in
(\mathbb T^1)^\infty$.

An important class of irreducible Schur--Weyl representations consists of
representations determined by sequences of Schur--Weyl embeddings that
are homogeneous in $N$, as defined below.

\begin{definition}\label{def:stat}
An irreducible Schur--Weyl
representation determined by a sequence $h=(h_0,h_1,h_2,{\ldots})\in
(\mathbb T^1)^\infty$ is called a stationary Schur--Weyl representation
if all $h_k$ coincide with the same vector of $\mathbb T^1$.
\end{definition}

Denote
\begin{eqnarray*}
\Todd=\{\tau=(\tau^{(1)},\tau^{(3)},\tau^{(5)},{\ldots} ):
\tau_n\in\Y_n^2,\;\tau^{(n)}\subset\tau^{(n+2)}\mbox{ for all }n,
\\
(\tau^{(n)},\tau^{(n+2)})\mbox{ is nice for sufficiently large }n\}.
\end{eqnarray*}
Thus $\Todd$ is the set of restrictions of two-row Young tableaux (regarded as
sequences of Young diagrams) to the {\it odd} levels such that for
sufficiently large $n$, the $(n+2)$th diagram is obtained from the
$n$th one by adding one cell to each row.

By definition, for $\tau\in\Todd$ with
$\tau^{(n)}=(\tau^{(n)}_1,\tau^{(n)}_2)$, the
sequence $\tau^{(n)}_1-\tau^{(n)}_2$
stabilizes to some value $k=1,3,5,{\ldots} $:
$\tau^{(n)}_1-\tau^{(n)}_2=k$ for sufficiently large $n$. Denoting the
corresponding subset of $\Todd$ by $\Todd_k$, we have
$$
\Todd=\bigcup_{j=1}^\infty\Todd_{2j-1}.
$$
The sets $\Todd_k$ and $\Todd$ are, obviously, countable. The set
$\Todd_k$ consists of tail-equivalent sequences.

Let $h=(h_0,h_1,{\ldots} )$ be the sequence that determines the
irreducible Schur--Weyl
representation $\Pi_k(h)$ under study; recall that $h_j$ determines the embedding
$\pi_{(k+j,j)}\hookrightarrow\pi_{(k+j+1,j+1)}$.
Consider an arbitrary
$\tau\in\Todd$. By the definition of $\Todd$, for sufficiently
large $j$, the pair $(\tau_{k+2j},\tau_{k+2j+2})$ is nice, so that we
may identify $H(\tau_{k+2j},\tau_{k+2j+2})$ with $H^{1,1}$ and assume that $h_{j}\in H(\tau_{k+2j},\tau_{k+2j+2})$. Let
$H_\tau^h$ be the {\it incomplete tensor product} \cite{vN}
of the spaces $H(\tau_{n},\tau_{n+2})$ determined by $h$, that is,
the completion of the set of all finite linear
combinations of simple tensor vectors $\bigotimes_{j=0}^\infty
v_{2j+1}$ with $v_n\in H(\tau_{n},\tau_{n+2})$ such that all but
finitely many of $v_{n}$
coincide with $h_{n}$.

In the even case, the argument is the same, with the space $\Teven$,
defined in a similar way,
in place of $\Todd$. However, in this case
we have an exceptional representation
$\Pi_0$ (which does not depend on $h$), which is
the inductive limit of the
representations with Young diagrams $(n,n)$.
This is the so-called ``discrete'' elementary representation
$D_{t_0}$, which is realized in the $l^2$
space spanned by all infinite Young tableaux tail-equivalent to the
``principal''
tableau $t_0$ with $1,3,5,{\ldots} $ in the first row and $2,4,6,{\ldots} $
in the second row; that is, the representation whose spectral measure
with respect to the Gelfand--Tsetlin subalgebra
is $\de_{t_0}$.

Summarizing the above discussion, we obtain the following proposition.

\begin{proposition}
Let $k\ge1$. Then, denoting by $H_k(h)$ the space of the
representation $\Pi_k(h)$, we have
$$
H_k(h)=\bigoplus_{\tau\in\Todd_k}H_\tau^h
$$
(however, the subspaces $H_\tau^h$ are not invariant under the action of
$\sinf$).
\end{proposition}

\section{Spectral measures of Schur--Weyl representations}
\label{tableaux}

Now we want to  construct a realization of the representation
$\Pi_k(h)$, $k\ge1$,
in the space $L^2(\T,\nu)$ for some measure $\nu$ on the space of all infinite
Young tableaux $\T$.

We have $h_j\in H(\la^{(j)},\la^{(j+1)})$, where $\la^{(j)}=(k+j,j)$.
Let $h_j=p_jt_{21}+q_jt_{12}$, where $\{t_{21},t_{12}\}$ is the standard
tableaux basis of $H(\la^{(j)},\la^{(j+1)})\simeq H^{(1,1)}$. By
unitarity, $p_j^2+q_j^2=1$.

Let $\Tproper$ be the set of finite Young tableaux $t\in\T_N$,
$N\in\mathbb N$, such that
the pair $([t]_{N-2},t)$ is not nice (that is, $N-1$ and $N$ lie in
the same row of $t$), and let
$\Tproper_k=\{t\in\Tproper:t\in\T_N(\la)\text{ with }\la_1-\la_2=k\}$.
Given
$t\in\Tproper_k$, consider the measure $\mu_t^h$ on $\T$ that is the
distribution of the following random walk on $\T$: we start from
$t\in\T_N$, and at each step, passing from the $n$th level to the
$(n+2)$th level, we choose the path corresponding to $t_{21}$ with
probability $\al_n=p_j^2$ and the path corresponding to $t_{12}$ with
probability $\beta_n=q_j^2$, where $j=[(n-k)/2]$.

\medskip
Let
$$
\mu^{(k)}=\sum_{t\in\Tproper_k}\mu_t.
$$

\begin{theorem}\label{th:Bernoulli}
The representation $\Pi_k(h)$ has a simple spectrum with respect to
the Gelfand--Zetlin algebra, and
\beq\label{isom}
\Pi_k(h)\simeq L^2(\T,\mu^{(k)}).
\eeq
\end{theorem}
\begin{proof}
The norm in $L^2(\T,\mu^{(k)})$
will be denoted by $\|\cdot\|$.
By definition, $\Pi_k(h)$ is an inductive limit of irreducible
representations $\pi_{(k)}$, $\pi_{(k+1,1)}$, $\pi_{(k+2,2)}$,{\ldots}.
The representation $\pi_{(k)}$ is
one-dimensional; to the only tableau $t$ with diagram
$(k)$  we associate the function $\phi_t=\de_t$, where $\de_t$ is the cylinder
function in $L^2(\T,\mu^{(k)})$ such that $\de_t(s)=1$ if $[s]_k=t$
and $\de_t(s)=0$ otherwise. Obviously, $\|\de_t\|$=1.
Now assume that we constructed functions
$\phi_t\in L^2(\T,\mu^{(i)})$ for all $t$ with diagram $(k+l,l)$.
Denoting $N=k+2l$ and
considering the restriction of the irreducible representation
$\pi=\pi_{(k+l+1,l+1)}$ of ${\mathfrak S}_{N+2}$ to ${\mathfrak
S}_{N}$, we have $\pi=\pi_{(k+l,l)}\oplus\pi'$, where
$\pi'=\pi_{(k+l+1,l-1)}+\pi_{(k+l-1,l+1)}$, the latter term missing if
$k<2$. Now let $t\in\T_{N+2}((k+l+1,l+1))$. If $[t]_N\in\T_{N}(\la)$ with
$\la=(k+l+1,l-1)$ or
$\la=(k+l-1,l+1)$, then $t\in\Tproper_k$ and we put $\phi_t=\de_t$. If
$[t]_N\in\T_N((k+l,l))$, then $\phi_{[t]_N}$ is already constructed
and we put
$$
\phi_t=\phi_{[t]_N}\cdot\begin{cases}
\frac1p\de_t&\mbox{if $N+1$ lies in the second row in $t$},\\
\frac1q\de_t&\mbox{if $N+1$ lies in the first row in $t$}.
\end{cases}
$$
It is not difficult to verify that the map $t\mapsto\phi_t$ thus defined
is an isometry $\pi_{(k+l,l)}\to L^2(\T,\mu^{(k)})$ and these maps
agree with the embeddings $\pi_{(k+l,l)}\hookrightarrow\pi_{(k+l+1,l+1)}$.
\end{proof}

In particular, for a stationary Schur--Weyl representation (see
Definition~\ref{def:stat}), we have
 $\al_j\equiv\al$, $\beta_j\equiv\beta$ for all $j$.
Thus a stationary Schur--Weyl representation is determined by a number
$p\in[-1,1]$, where $h_j=pt_{21}+qt_{12}$ and $q=\sqrt{1-p^2}$
(it suffices to consider only positive $q$, because the embeddings
determined by parameters $(p,q)$ and $(-p,-q)$ are obviously
equivalent). The corresponding measure $\mu^{(k)}$ on the space of
infinite Young tableaux is the distribution of a stationary
(``Bernoulli'') random walk.

\medskip\noindent{\bf Example 1.}
If each $h_j$ coincides with $t_{21}$ (or with $t_{21}$), i.e.,
$p\in\{0,1\}$, then the
measure $\mu^{(k)}$ is the $\de$-measure at an infinite Young tableau, and
the corresponding representation $\Pi_k(h)$ is the
discrete elementary representation of $\sinf$ associated with this
tableau.

\medskip\noindent{\bf Remark 1.} The measure $\mu^{(k)}$ on
the set of infinite Young tableaux is not central. It is $\si$-finite,
continuous, and ergodic with respect to the tail equivalence relation
on partitions (since the corresponding representation is irreducible); the representation under study acts in the scalar $L^2$
space over this measure.

\smallskip\noindent
{\bf Remark 2.} The action of the infinite symmetric group $\sinf$ in
the space $L^2(\T,\mu^{(k)})$ providing the isomorphism~\eqref{isom}
does not coincide with the standard action of permutations on Young
tableaux determined by Young's
orthogonal form.

\section{The main example: tensor Schur--Weyl representations}
\label{product}

There is a distinguished Schur--Weyl embedding that agrees with the
tensor product structure of the space $(\C^2)^{\otimes N}$. Namely,
observe that $\C\otimes\C=M_1\oplus
M_3$ as $SL(2,\C)$-modules, where $M_3$ is the three-dimensional irreducible $SL(2,\C)$-module
and $M_1$ is the trivial one-dimensional $SL(2,\C)$-module. Thus we may
embed $(\C^2)^{\otimes N}$ into
$$
(\C^2)^{\otimes(N+2)}=(\C^2)^{\otimes N}\otimes(\C^2)^{\otimes2}=
(\C^2)^{\otimes N}\otimes(M_1\oplus M_3)
$$
along this one-dimensional representation.

In other words, this is the unique
Schur--Weyl embedding that has the form
$$
(\C^2)^{\otimes N}\ni
v\mapsto v\otimes v_0\in(\C^2)^{\otimes(N+2)}
$$ for some $v_0\in\C^2$. It is easy to see
that, denoting the standard basis of $\C^2$ by
$\{e_1,e_2\}$, we have
$v_0=\frac1{\sqrt2}(e_1\otimes e_2-e_2\otimes e_1)$; this is the
unique (up to a constant)
$SL(2,\C)$-invariant vector in $\C^4$.

\begin{definition}
The embedding described above will be called the tensor Schur--Weyl
embedding.
\end{definition}

Just another way to describe the tensor Schur--Weyl embedding is to say that
the image of $(\C^2)^{\otimes N}$ in $(\C^2)^{\otimes(N+2)}$ should
lie in the eigenspace of
the last Coxeter transposition $\si_{N+1}=(N+1,N+2)$ with eigenvalue $-1$.

\begin{proposition}
For the tensor Schur--Weyl embedding described above, we have
$$
p=-\sqrt{\frac{r-1}{2r}},\qquad q=\sqrt{\frac{r+1}{2r}},
$$
so that the
corresponding Bernoulli measure $\mu_t$ has the weights
$$
\al=\frac{{r-1}}{{2r}}, \qquad\beta=\frac{{r+1}}{{2r}};
$$
here $r=k+1$ for $t\in\Tproper_k$.
\end{proposition}

\begin{proof}
Let $\la=(k+n,n)$, $\mu=(k+n+1,n+1)$, and $h\in H(\la,\mu)$
be the vector determining the tensor Schur--Weyl embedding in
these components, so that $h=pt_{21}+qt_{12}$.
Recall that Young's orthogonal form (see, e.g., \cite{James}) says
that the Coxeter transposition $\si_{N+1}$, where $N=2n+k$, acts in
$H(\la,\mu)$ and has the following matrix in the basis
$\{t_{21},t_{12}\}$:
$$
\begin{pmatrix}
r^{-1}&\sqrt{1-r^{-2}}\\
\sqrt{1-r^{-2}}&-r^{-1},
\end{pmatrix}
$$
where the axial distance $r$ (defined as $c_{N+2}-c_{N+1}$, where
$c_j$ is the content of the cell containing $j$) in our case is equal
to $k+1$. The proposition now follows from the fact that the image of
$(\C^2)^{\otimes N}$ in $(\C^2)^{\otimes(N+2)}$ should
lie in the eigenspace of
$\si_{N+1}=(N+1,N+2)$ with eigenvalue $-1$.
\end{proof}

Obviously, the tensor Schur--Weyl representation can be realized in the incomplete tensor product
$$
(\C^4)^{\otimes\infty}_{v_0},\quad
v_0=\frac1{\sqrt2}(e_1\otimes e_2-e_2\otimes e_1)\in\C^4.
$$
It follows that in the space of the tensor Schur--Weyl representation
we have also an action of the UHF algebra ${\cal
G}=\varinjlim\operatorname{Mat}_{4^n}(\C)$ (the Glimm algebra of type~$2^\infty$; see, e.g., \cite{Pedersen}),
thus obtaining a new representation of this algebra.
One can show that this is the only Schur--Weyl representation with
this property.

There is a natural question: given a vector
$w\in(\C^4)^{\otimes\infty}_{v_0}$, determine in what primary component
${\cal H}_k=\Pi_k\otimes M_{k+1}$ of the
corresponding ``spin'' decomposition~\eqref{thdecomp} it lies.
The answer is as
follows. By the definition of an incomplete tensor product, we have
$w=u\otimes v_0\otimes v_0\otimes{\ldots} $, where
$u\in(\C^4)^{\otimes N}$ for some finite $N$. Then $w$ has the same
spin $k$ as the finite vector $u$ does according to
decomposition~\eqref{finodd} or~\eqref{fineven}. To find this $k$, one
may write $(\C^4)^{\otimes N}$ as $(\C^2)^{\otimes 2N}$ and then use
the results on tensor representations of the finite symmetric groups
from \cite[Section~5]{markov}.

In the class of generalized Schur--Weyl
representations (see Definition~\ref{generalized}) there are many
other tensor representations. Namely, the ``jump factor'' $(\C^2)^{\otimes k}$
contains the one-dimensional representation $M_1$ of $SL(2,\C)$ with
multiplicity equal to the Catalan number $C_k$. Choosing any vector
from this $C_k$-dimensional $SL(2,\C)$-invariant subspace as $v_0$ (in
general, we may choose
different vectors at different steps) and
constructing the corresponding incomplete tensor product, we obtain a
tensor generalized Schur--Weyl
representation.

\section{The Coxeter--Laplace operator in Schur--Weyl representations}
\label{laplacian}

In this section we study the so-called Coxeter--Laplace operator in
Schur--Weyl representations. But first let us briefly describe the general
setting.

In the general theory of random walks,
the Laplacian of the random walk on a finite group $G$ with probability
measure $\mu$ is the operator $E-\sum_g \mu(g)L_g$
in the group algebra of $G$ (or in
the space where a representation of $G$ is defined),
where the sum is over all
$g\in G$ with $\mu(g)>0$, $L_g$ is the operator of left
multiplication by $g$, and $E$ is the identity operator.
Let $G$ be the symmetric
group ${\frak S}_N$ and $\mu$ be the uniform measure
on the set of Coxeter generators $\sigma_k=(k,k+1)$, $k=1, \ldots, N$, where we
set $\sigma_N=(N,1)$. The
corresponding Laplacian has the form
$$
L_N=e-\frac1N\sum_{k=1}^N \sigma_k.
$$
We call it the {\it periodic Coxeter
Laplacian}, or the {\it Coxeter--Laplace operator}.
It follows from the Schur--Weyl duality
that if $\pi_N$ is the representation of
$\sN$
in the tensor product $(\C^2)^{\otimes N}$ by permutations of factors,
then the operator $\pi_N(L_N)$ is related to the Hamiltonian of the XXX
Heisenberg model on the periodic one-dimensional lattice with $N$
sites (see, e.g., \cite{Iz,FT}) by the formula $H=\frac J4(2L-N)$,
where $J>0$ corresponds to
the ferromagnetic case, and
$J<0$ corresponds to the antiferromagnetic case
(see also \cite{ferro, MTV}).

Thus we have the following natural problem.
For a given (irreducible or not) representation of the
symmetric group, find the eigenvalues and eigenfunctions of the Coxeter
Laplacian. Usually, one considers the spectrum of operators in $L^2(G)$.
In our case,
it is natural, both from the point of view of representation theory
and applications to physics,
to consider the asymptotics of the Coxeter Laplacian $L_N$ and its
spectrum as $N\to\infty$.
The case most important for applications (in particular, for the
Heisenberg model) is that
of representations of ${\frak S}_N$ corresponding to Young diagrams with
finitely many rows; e.g., at most two rows.
Besides, there are two different asymptotic modes,
``ferromagnetic'' mode and ``antiferromagnetic'' mode. The first case is
easier: it means considering representations corresponding to Young
diagrams with fixed second row and growing first row; for asymptotic
results in this case, see \cite{ferro}. But the most interesting case
is when both rows grow; namely, we are interested in the asymptotics
of the largest eigenvalue of $L_N$, which corresponds to the ground
energy of the Heisenberg antiferromagnet.
Some related results  are
obtained in \cite{antiferro} and cited below.

Here we study the asymptotic behavior of the
Coxeter--Laplace operator in the limit of the representations $\pi_N$
corresponding to a stationary chain of Schur--Weyl embeddings. We are
especially interested in the ``leading'' components $\Pi_0$ and $\Pi_1$
for the even and odd case, respectively, because they correspond to
the ground state of the antiferromagnetic Heisenberg chain (see
\cite{antiferro}). But $\Pi_0$ is just the discrete representation
$D_{t_0}$ corresponding to the principal tableau $t_0$ (see Example~1), and it is
shown in \cite{antiferro} that the weak limit of the operators
$\frac1N\pi_N(L_N)$ in this representation is the scalar operator with
constant $5/4$.

\begin{proposition}\label{prop:p}
The weak limit of the operators $\frac1NL_N$ in the
leading component $\Pi_1$ of the
stationary Schur--Weyl
representation of $\sinf$ with parameters $p$, $q=\sqrt{1-p^2}$ is the
identity operator $\phi(p)E$, where
$$
\phi(p)=\frac{13}{12}+\frac86p^4-\frac76p^2-\frac{\sqrt{3}}2p\sqrt{1-p^2}.
$$
\end{proposition}
\begin{proof}
The representation $\Pi_1$ is the inductive limit of the irreducible
representations $\pi_{(1)}$, $\pi_{(2,1)}$, $\pi_{(3,2)}$,{\ldots} of
the symmetric groups ${\mathfrak S}_1$, ${\mathfrak S}_3$, ${\mathfrak
S}_5$,{\ldots} under the stationary sequence $i_{2k+1}:{\mathfrak
S}_{2k+1}\hookrightarrow{\mathfrak S}_{2k+3}$ of
Schur--Weyl embeddings determined by $p$.
Thus a basis of $\Pi_1$ consists of the images in $\Pi_1$ of the finite Young
tableaux $t\in\T((k+1,k))$, $k=0,1,{\ldots} $. Let $t,s$ be such
tableaux and $n=2k+1$. Obviously,
$$
\lim_{N\to\infty}\frac1N(L_Nt,s)=\lim_{N\to\infty}\frac1N
\sum_{j=n+1}^{N-1}(\si_jt,s).
$$
Further, since the Schur--Weyl embeddings under consideration are
stationary, we have
$((\si_{j+2}+\si_{j+3})t,s)=((\si_{j}+\si_{j+1})t,s)$ for all $j>n$,
whence
$$
\lim_{N\to\infty}\frac1N(L_Nt,s)=((\si_{n+1}+\si_{n+2})t,s),
$$
the latter expression depending only on the image of $t$ and $s$ in
$\pi_{(k+3,k+2)}$ under the embedding $i=i_{n+2}\circ i_n$.
Now we have
$$
i(t)=p^2t_{pp}+q^2t_{qq}+pq(t_{pq}+t_{qp}),
$$
where $t_{pp}$ is the tableau obtained from $t$ by putting $n+1$ into
the second row and $n+2$ into the first row, and then $n+3$ into the
second row and $n+4$ into the first row; $t_{pq}$ is the tableau
obtained from $t$ by putting $n+1$ into
the second row and $n+2$ into the first row, and then $n+3$ into the
first row and $n+4$ into the second row; etc. A similar formula holds
for $s$, and the rest follows from straightforward calculations based
on Young's orthogonal form.
\end{proof}

In particular, for the tensor Schur--Weyl embedding we have
$\phi(-1/2)=5/4$ (in fact, one can prove that $\frac1N L_N$
has the same limit in all
components $\Pi_k$, $k=0,1,2,{\ldots} $). The maximum possible
value $\phi(p)$ corresponds to $p=-0.95543{\ldots}$ and is equal to
$c_{SW}=1.3736684{\ldots}$. The limiting operators for other natural
representations of $\sinf$ are found in \cite{antiferro}; they are
also scalar, and the corresponding constants are smaller:
\begin{itemize}
\item $1.3736684{\ldots}$ for the stationary Schur--Weyl representation with
$p=-0.95543{\ldots}$;
\item $1.25$ for the discrete
representation $D_{t_0}$ (see Example~1);
\item$1$ for the
representation induced from the identity representation of the Young
subgroup ${\mathfrak S}_{\{1,3,5,{\ldots} \}}\times{\mathfrak
S}_{\{2,4,5,{\ldots} \}}$;
\item $0.5$ for the factor representation with
Thoma parameters $\al=(1/2,1/2)$, $\beta=0$.
\end{itemize}

As proved in
\cite{yangyang2},
the maximum eigenvalue $\la_N$ of
$\pi_N(L_N)$ satisfies $\frac{\la_N}N\to c_{\rm max}=2\log2=1.38629436{\ldots}$.
Thus the ``deficiency'' of $c_{SW}$ as compared with $c_{\rm max}$ is
just about $0.0126$.
It would be very interesting to find a
representation of the infinite symmetric group in which the limit of
the Coxeter--Laplace operator has the eigenvalue $c_{\rm max}$. A step in this
direction is
Proposition~\ref{prop:apprmax} below, which shows that we can construct a
representation with an eigenvalue arbitrarily close to $c_{\rm max}$.

\begin{proposition}\label{prop:apprmax}
For every $\eps>0$ there exists $k\in\mathbb N$ and a stationary sequence
of generalized Schur--Weyl embeddings
$$
{\mathfrak S}_{1}\subset{\mathfrak S}_{1+2k}\subset{\mathfrak
S}_{1+4k}\subset{\ldots}
$$
such that the weak limit of the operators $\frac1NL_N$ in the
leading component of the
corresponding generalized Schur--Weyl
representation of $\sinf$ is a scalar
operator $cE$ with $c>c_{\rm max}-\eps$.
\end{proposition}
\begin{proof}
Let $t,s\in\T((mk+1,mk))$, and let  $N=2nk+1$ for $n>m$. Arguing as
in the proof of Proposition~\ref{prop:p}, we have
\begin{eqnarray*}
\lim_{N\to\infty}\frac1N(L_Nt,s)&=&\frac1{2k}\big((\si_{mk+2}+{\ldots}
+\si_{(m+2)k+1})t,s\big)\\
&=&\frac1{2k}\big((T^{mk}L_{2k+1})t,s\big)
+\frac1{2k}\big(R_kt,s\big),
\end{eqnarray*}
where $R_k=\si_{(m+2)k+1}-\si_{mk+1}-(mk+1,(m+2)k+1)$ and $T$ is the
endomorphism of $\sinf$  defined by the formulas $(Tg)(1)=1$,
$(Tg)(i)=g(i-1)+1$ for $i>1$ (the infinite shift).
Now let $v_{\rm max}$ be the eigenvector of $L_{2k+1}$ corresponding
to the largest eigenvalue $\la_{\rm max}^{(k)}$. As proved in \cite{antiferro}, $v_{\rm
max}$ lies in the irreducible  representation $\pi_{(k+1,k)}$. Choose an embedding
$\pi_{(1)}\hookrightarrow\pi_{(k+1,k)}$
such that the only tableau in $\T((1))$ goes to $v_{\rm max}$.
Then, since the generalized embedding is stationary, we have
$$
\frac1{2k}\big((T^{mk}L_{2k+1})t,s\big)=
\frac1{2k}\big(L_{2k+1}[t]_{2k+2},[s]_{2k+2}\big)=\frac{\la_{\rm
max}^{(k)}}{2k},
$$
and the proposition follows.
\end{proof}

The calculation in \cite{yangyang2} of the limit value $c_{\rm max}$,
based on Bethe ansatz, is very indirect, and, in particular, relies on considering a more
general model that has no interpretation in terms of the
symmetric groups. However, studying the
asymptotics of the spectrum of the Coxeter--Laplace operator  is a
natural problem for the theory of symmetric groups. Thus a challenge
is to obtain a representation-theoretic proof of this result.

\end{document}